\documentclass[8pt]{article}
\usepackage{indentfirst,latexsym,bm}
\usepackage{amsfonts}
\usepackage{times}
\usepackage{amssymb}
\usepackage[leqno]{amsmath}
\usepackage{dsfont}
\usepackage{color,xcolor}
\usepackage{amsthm}
\usepackage{hyperref}
\begin{document}
\title{Pre-modular  fusion categories of small global dimensions}
\author{Zhiqiang Yu}
\date{}
\maketitle

\newtheorem{theo}{Theorem}[section]
\newtheorem{prop}[theo]{Proposition}
\newtheorem{lemm}[theo]{Lemma}
\newtheorem{coro}[theo]{Corollary}
\theoremstyle{definition}
\newtheorem{conj}[theo]{Conjecture}
\newtheorem{defi}[theo]{Definition}
\newtheorem{exam}[theo]{Example}
\newtheorem{ques}[theo]{Question}
\newtheorem{remk}[theo]{Remark}

\newcommand{\A}{\mathcal{A}}
\newcommand{\B}{\mathcal{B}}
\newcommand{\C}{\mathcal{C}}
\newcommand{\D}{\mathcal{D}}
\newcommand{\E}{\mathcal{E}}
\newcommand{\I}{\mathcal{I}}
\newcommand{\FC}{\mathbb{C}}
\newcommand{\FQ}{\mathbb{Q}}
\newcommand{\M}{\mathcal{M}}
\newcommand{\N}{\mathcal{N}}
\newcommand{\Q}{\mathcal{O}}
\newcommand{\W}{\mathcal{W}}
\newcommand{\Y}{\mathcal{Z}}
\newcommand{\Z}{\mathbb{Z}}
\theoremstyle{plain}

\abstract
We first prove an analogue  of Lagrange theorem for global dimensions of fusion categories, then we give a complete  classifications of pre-modular fusion categories of integer global dimensions less than or equal to $10$.


\section{Introduction}
Throughout this paper, let $\FC$ be the field of complex numbers, $\FC^{*}:=\FC\backslash {\{0}\}$, $\FQ$ and $\overline{\FQ}$   denote the fields of rationals and its algebraic closure, respectively. For $r\in \mathbb{N}$, let $\Z_r:=\Z/r$.

A fusion category $\C$ is a semisimple $\FC$-linear finite abelian rigid monoidal category. Fusion category $\C$ spherical, if $\C$ admits a pivotal structure $j$, which is a natural isomorphism from identity tensor functor $\text{id}_\C$ to double dual tensor functor $(-)^{**}$, and the pivotal structure $j$ satisfies $\text{dim}_j(X)=\text{dim}_j(X^*)$ for any object $X$ of $\C$, here $\text{dim}_j(X)$ is the quantum dimension of object $X$ determined by $j$, see section \ref{preliminaries} for definition.

We know that all fusion categories of prime FP-dimensions are pointed \cite[Corollary 8.30]{ENO1}. However, this is not the case for global dimensions.  In \cite[Example 5.1.2]{O3},  Ostrik classified all spherical fusion categories of integer global dimensions less than or equal to $5$. These fusion categories are  either pointed, or tensor equivalent to an Ising category $\I$ or equivalent to a Deligne tensor product $YL\boxtimes \overline{YL}$, where $YL$ is a Yang-Lee fusion category and $ \overline{YL}$  is a Galois conjugate of $YL$. Obviously, $YL\boxtimes \overline{YL}$ is not pointed and  $\text{dim}(YL\boxtimes \overline{YL})=5$. Therefore, this is a non-trivial task to classify spherical fusion categories of  given  global dimensions. Meanwhile, for a given global dimension, it follows from \cite[Theorem 1.1.1]{O3} that there are finitely many tensor equivalence classes  of spherical fusion categories, then it is doable to classify spherical fusion categories of small global dimensions.

When classifying spherical fusion categories by global dimensions, one of the main difficulties is   to restrict the rank of these fusion categories, and there is no general method, see \cite[Lemma 4.2.2]{O3}. Recall that a spherical fusion category $\C$ is a pre-modular fusion category  if $\C$ is braided. For pre-modular fusion category $\C$, we have the $S$-matrix and $T$-matrix (see section \ref{preliminaries} for definition), which reflect some important aspects of  pre-modular fusion category $\C$,   moreover they enjoy interesting arithmetic properties \cite{BNRW,DongLinNg,NS}, and  see \cite{BGNPRW,BNRW,BPRZ,Green,RSW} for applications to the  classifications of (super-)modular fusion categories of low ranks. So, in this paper, we turn our attentions to pre-modular fusion categories.

It is well-known that the Lagrange theorem for FP-dimension holds for fusion categories
\cite[Proposition 8.15]{ENO1}: Given a  fusion category $\C$ and fusion subcategory $\D\subseteq\C$, the ratio $\frac{\text{FPdim}(\C)}{\text{FPdim}(\D)}$ is an algebraic integer. In this paper,
 we first prove  an  analogue of Lagrange theorem for global dimension. That is,
 \begin{theo}[Theorem \ref{Lagrange}]Let  $\C$ be a fusion category. If $\D\subseteq\C$ is a fusion subcategory, then $\frac{\text{dim}(\C)}{\text{dim}(\D)}$ is an algebraic integer.
 \end{theo}
Then we can use  Theorem \ref{Lagrange} to restrict global dimensions of fusion subcategories of $\C$, moreover we think Theorem \ref{Lagrange} can be used in the future research of fusion categories. However, we need to point it out that $\frac{\text{FPdim}(\C)}{\text{FPdim}(\D)}\neq\frac{\text{dim}(\C)}{\text{dim}(\D)}$ in general, see Remark \ref{nonequratios} for some examples. Together with the techniques developed in \cite{BNRW,BPRZ,O2,O3},  we obtain some  classifications results on pre-modular fusion categories of small integer global dimensions.

This paper is organized as follows.  In section \ref{preliminaries}, we recall some basic notions and notations of    fusion categories, such as pre-modular fusion categories, global dimensions, formal codegrees and $d$-numbers. In  section \ref{section3},   we   prove the Lagrange theorem   for global dimension of fusion categories in Theorem \ref{Lagrange}. In subsection \ref{subsection4.1}, we show that pre-modular fusion categories of   dimension  $7$ are pointed in Theorem \ref{moddimen7}. In subsection \ref{subsection4.2}, we classify pre-modular fusion categories of global dimension $8$, $9$ and $10$ in Corollary \ref{predimen8}, Theorem \ref{dimension9} and  Corollary \ref{corodimen10}, respectively. In subsection \ref{subsection4.3}, we   show that spherical fusion categories of global dimension $6$ are weakly integral, see Theorem \ref{spherical6}.

\section{Preliminaries}\label{preliminaries}
\subsection{Spherical fusion category}
Given a fusion category $\C$, let $\Q(\C)$ be the set of isomorphism classes of simple objects of $\C$. The cardinal of $\Q(\C)$ is called rank of $\C$, and will denoted by $\text{rank}(\C)$.
It is well-known that there is a unique ring  homomorphism  FPdim(-)  from the Grothendieck ring $\text{Gr}(\C)$ of $\C$ to $\FC$ such that $\text{FPdim}(X)\geq1$ is an algebraic integer for all objects $X\in\Q(\C)$
\cite[Theorem 8.6]{ENO1}, and $\text{FPdim}(X)$ is called the Frobenius-Perron dimension of object $X$. The Frobenius-Perron dimension $\text{FPdim}(\C)$ of fusion category $\C$ is defined by
\begin{align}
\text{FPdim}(\C):=\sum_{X\in\Q(\C)}\text{FPdim}(X)^2.
\end{align}

A fusion category $\C$ is weakly integral, if $\text{FPdim}(\C)\in\Z$; $\C$ is integral, if  $\text{FPdim}(X)\in\Z$ for all $X\in\Q(\C)$. A fusion category $\C$ is pointed  if and only if all simple objects of $\C$ have FP-dimension $1$, so $\C\cong \text{Vec}_G^\omega$ \cite{ENO1}, where $\text{Vec}_G^\omega$ is the category of $G$-graded finite-dimension vector spaces over $\FC$, $\omega\in Z^3(G,\FC^*)$ is a $3$-cocycle. We use $\C_\text{pt}$ and $\C_\text{int}$ to denote the maximal pointed fusion subcategory and  the maximal integral fusion subcategory of $\C$, respectively.

Assume that $X\in\C$ is  an object, let  $(X^*, \text{ev}_X,\text{coev}_X)$ be a left dual object of $X$. That is, there exist natural morphisms $\text{coev}_X:I\to X\otimes X^*$ and $\text{ev}_X:X^*\otimes X\to I$ satisfying the following equations
\begin{align*}
\text{id}_X\otimes \text{ev}_X \circ \text{coev}_X\otimes \text{id}_X=\text{id}_X,\quad \text{ev}_X \otimes \text{id}_{X^*} \circ \text{id}_{X^*}\otimes \text{coev}_X=\text{id}_{X^*}.
\end{align*}
Here, we suppress the associativity and   unit constraints  of $\C$.

In a fusion category $\C$, since $X\cong X^{**}$ and $\text{Hom}_\C(X,X)\cong \FC$ for   object  $X\in\Q(\C)$, up to scalar, there is a unique isomorphism $\alpha_X:X\overset{\sim}{\to} X^{**}$. Then for any morphism $f:X\to X$, following \cite{ENO1,Mu1},  we define the (categorical) trace of $\alpha_X\circ f$ as the following scalar
\begin{align}
\text{Tr}(\alpha_X\circ f):=\text{ev}_{X^*}\circ  (\alpha_X\circ f)\otimes \text{id}_{X^*}\circ\text{ coev}_X: I\to I.
\end{align}
Then the square norm of simple object $X$ \cite{Mu1} is defined as
\begin{align}
|X|^2:=\text{Tr}(\alpha_X) \text{Tr}((\alpha_X^*)^{-1}),
\end{align}
and the global (or, categorical) dimension of   fusion category $\C$ as
\begin{align}
\text{dim}(\C):=\sum_{X\in\Q(\C)}|X|^2.
\end{align}

It is easy to see that global dimension $\text{dim}(\C)$ (indeed, $|X|^2$) is independent of the choice of isomorphisms ${\{\alpha_X|X\in\Q(\C)}\}$, moreover, $\text{dim}(\C)\geq1$ is an algebraic integer
\cite[Theorem 2.3]{ENO1}.
It follows from \cite[Proposition 8.22]{ENO1} that the ratio $\frac{\text{dim}(\C)}{\text{FPdim}(\C)}\leq1$ is an algebraic integer. In addition, $\text{dim}(\C)>\frac{4}{3}$  if $\C$ is a non-trivial spherical fusion category \cite[Theorem 1.1.2]{O3}. A fusion category $\C$ is said to be pseudo-unitary if $\text{dim}(\C)=\text{FPdim}(\C)$.  For more properties of global dimension, we refer the readers to references \cite{EGNO,ENO1,Mu1,O3}.

Let $\C$ be a pivotal fusion category with a pivotal structure $j$, that is, $j$ is a natural isomorphism from identity tensor functor $\text{id}_\C$ to the double dual tensor functor $(-)^{**}$. Then we define $\text{dim}_j(X):=\text{Tr}(j_X)$,  the quantum (or categorical) dimension of $X$  determined by $j$, moreover $\text{dim}_j(-)$ induces a homomorphism from $\text{Gr}(\C)$ to $\FC$ \cite[Proposition 4.7.12]{EGNO}.  A direct computation shows that $j_{X^*}=((j_X)^*)^{-1}$
\cite[Exercise 4.7.9]{EGNO}, so $\overline{\text{dim}_j(X)}=\text{dim}_j(X^*)$ by \cite[Proposition 2.9]{ENO1} and $|X|^2=\text{dim}_j(X)\text{dim}_j(X^*)$. Hence,
\begin{align}
\text{dim}(\C)=\sum_{X\in\Q(\C)}\text{dim}_j(X)\text{dim}_j(X^*)
=\sum_{X\in\Q(\C)}\text{dim}_j(X)\overline{\text{dim}_j(X)}.
\end{align}
It was conjectured in \cite[Conjecture 2.8]{ENO1} that every fusion category admits a pivotal structure. A pivotal fusion category $\C$ is  spherical, if $\C$ admits a pivotal structure $j$ such that $\text{dim}_j(X)=\text{dim}_j(X^*)$ for all objects $X$ of $\C$. Thus, for a spherical fusion category $\C$, we have
 \begin{align}
\text{dim}(\C)=\sum_{X\in\Q(\C)}\text{dim}_j(X)^2.
\end{align}

 We fix a spherical structure  $j$ of $\C$,   and we use $\text{dim}(X)$ instead of  $\text{dim}_j(X)$ to denote the quantum dimension of $X$ below.
Given an arbitrary spherical fusion category $\C$,  we can consider the twist (or Galois conjugate) $\C^\sigma$ of $\C$, where $\sigma\in\text{Gal}(\overline{\mathbb{Q}}/\mathbb{Q})$. More precisely, $\C^\sigma$ is a fusion category with the same monoidal functor $\otimes$ as $\C$, but the associator of $\C^\sigma$ is obtained by composing  the one of $\C$ with automorphism $\sigma$. Moreover, $\text{dim}(\C^\sigma)=\sigma(\text{dim}(\C))$.

We say that a fusion category $\C$ is simple, if $\C$ does not contain any fusion subcategory  other than $\C$ and $\text{Vec}$. For example, pointed fusion category $\text{Vec}^\omega_{\Z_p}$ is simple for any prime $p$; and fusion category $\text{Rep}(G)$ is simple if and only if  $G$ is a finite simple group, where $\text{Rep}(G)$ is the category of finite-dimensional representations of $G$  over $\FC$.

\subsection{Pre-modular fusion category}\label{subsection2.2}
A fusion category $\C$ is a braided  fusion category if $\C$ admits a braiding $c$. Specifically, for any objects $X,Y,Z\in\C$, there exists a natural isomorphism $c_{X,Y}:X\otimes Y\overset{\sim}{\to} Y\otimes X$, which satisfies $c_{X,I}=c_{I,X}=\text{id}_X$, $c_{X\otimes Y,Z}=c_{X,Z}\otimes \text{id}_Y \circ \text{id}_X\otimes c_{Y,Z}$, $c_{Z,X\otimes Y}=\text{id}_X\otimes c_{Z,Y}\circ c_{Z,X}\otimes \text{id}_Y$, here we suppress the associativity isomorphism of $\C$.

Let $\D$ be a fusion subcategory of braided fusion category $\C$, the centralizer of $\D$ in $\C$ is the following fusion subcategory
\begin{align*}
\D_\C'={\{X\in\C|c_{Y,X} c_{X,Y}=\text{id}_{X\otimes Y},\forall Y\in\D}\}.
\end{align*}
 We call $\C':=\C_\C'$ the M\"{u}ger center of $\C$ \cite{Mu}. $\D$ is symmetric if $\D\subseteq\D_\C'$; in addition,  a symmetric fusion category $\D$ is Tannakian if $\D\cong\text{Rep}(G)$, where the braiding of $\text{Rep}(G)$ is given by the  reflection of vector spaces. Moreover, if $\C'=\text{Vec}$, then   $\C$ is said to be non-degenerate, where $\text{Vec}$ is the category of finite-dimensional vectors spaces over $\FC$.

A braided fusion category $\C$ is a pre-modular (or  ribbon) category if $\C$ is spherical.
Given a pre-modular fusion category $\C$, we can define the $S$-matrix $S=(s_{X,Y})_{X,Y\in\Q(\C)}$ and $T$-matrix $T=(T_{X,Y})_{X,Y\in\Q(\C)}$ of $\C$ \cite{BK,EGNO}. Explicitly, $s_{X,Y}=\text{tr}(c_{Y,X}c_{X,Y})$ and $T_{X,Y}=\delta_{X,Y}\theta_X$ for all $X,Y\in\Q(\C)$, where $\theta$ is the ribbon structure of $\C$. Following \cite{BK}, we use $\C(\mathfrak{g},q,k)$ to denote the pre-modular fusion category obtained from representation category of quantum group $U_q(\mathfrak{g})$, where $k$ is a positive integer and $q^2$ is a root of unity of order $k$.

Consequently, a pre-modular fusion category $\C$ is modular if and only if its $S$-matrix is non-degenerate \cite{BK,DrGNO2,EGNO,Mu}, equivalently its M\"{u}ger center $\C'=\text{Vec}$. It is well-known that pointed modular fusion categories are in bijective correspondence with metric groups,
 see \cite[Appendix A]{DrGNO2}. We  denote by $\C(G,\eta)$ the  pointed modular fusion category determined by metric group $(G,\eta)$ below, where $\eta$ is a non-degenerate quadratic form on $G$.

 Meanwhile, a pre-modular fusion category $\C$ is super-modular, if $\C'\cong \text{sVec}$, where $\text{sVec}$ is the category of finite-dimensional super-vectors spaces over $\FC$. Throughout this paper, we assume $\text{dim}(\chi)=1$ and $\theta_\chi=-1$, where $\text{sVec}=\langle\chi\rangle$. It follows from \cite[Lemma 5.4]{Mu} or \cite[Lemma 3.28]{DrGNO2} that $\text{rank}(\C)$ must be even if $\C'=\text{sVec}$. In addition, given a super-modular fusion category $\C$, based on the partition of set $\Q(\C)=\Pi_0\cup \Pi_1$ (see \cite{BGNPRW}), we have
  \begin{align}\label{halfdimen}
 \frac{\dim(\C)}{2}=\sum_{X\in\Pi_0}\dim(X)^2,
\end{align}
and   one can define the so-called naive fusion rule: for arbitrary simple objects $X, Y, Z\in\Pi_0$,
\begin{align}
 \widehat{N}_{X,Y}^Z:=\text{dim}_{\FC}(\text{Hom}_\C(X\otimes Y,Z))+\text{dim}_{\FC}(\text{Hom}_\C(X\otimes Y,\chi\otimes Z)).
\end{align}

Moreover, there is a non-degenerate symmetric matrix $\widehat{S}$ such that the $S$-matrix of $\C$ is
\begin{align*}
S=\left(
    \begin{array}{cc}
      1 & 1 \\
      1 & 1 \\
    \end{array}
  \right)\otimes\widehat{S}=
\left(
                                       \begin{array}{cc}
                                         \widehat{S} & \widehat{S} \\
                                         \widehat{S} & \widehat{S} \\
                                       \end{array}
                                     \right),
\end{align*}
and $\widehat{S}$ has orthogonal rows,
see \cite[Proposition 2.7]{BGNPRW} for more properties of $\widehat{S}$ and \cite{Yu} for some applications.
 And some techniques for classifying super-modular category $\C$ are established in \cite{BPRZ}, which are similar to that of modular fusion categories \cite{BNRW,EGNO}. Recall that the  Galois group $\text{Gal}(\C):=\text{Gal}(\FQ(\widehat{S})/\FQ)$  is an abelian group, where $\FQ(\widehat{S})$ is the extension field of $\FQ$ by the coefficients of $\widehat{S}$. Then, for any $\sigma\in \text{Gal}(\C)$, we have
\begin{align}
\sigma(\frac{s_{X,Y}}{s_{I,Y}})
=\frac{s_{X,\hat{\sigma}(Y)}}{s_{I,\hat{\sigma}(Y)}}, ~\forall X,Y\in\Pi_0.
\end{align}
where $\hat{\sigma}$ is the unique  permutation  on set $\Pi_0$, which is induced by $\sigma$. Since the morphism $\sigma\mapsto\hat{\sigma}$ is a well-defined group isomorphism, sometime we also use $\text{Gal}(\C)$ to denote the image of $\text{Gal}(\C)$.

Given a pre-modular fusion category $\C$,  let $X,Y\in\Q(\C)$, then map $h_X(Y):=\frac{s_{X,Y}}{s_{I,X}}$ defines a homomorphism from Grothendieck ring $\text{Gr}(\C)$ to $\FC$ \cite[Proposition 8.13.11]{EGNO}. In particular, if $\C$ is modular, then  the Verlinde formula \cite{BK,EGNO}  says that
the set ${\{h_X|\forall X\in\Q(\C)}\}$ is a complete set of homomorphisms from Grothendieck ring $\text{Gr}(\C)$ to $\FC$.
 Thus, if $\text{dim}(X)^2\in\Z$ for all $X\in\Q(\C)$, then $\C$ is   weakly integral, as the rational number  $\text{FPdim}(\C)=\frac{\text{dim}(\C)}{\text{dim}(Y)^2}$ (for certain object $Y\in\Q(\C)$) is an algebraic integer.
\begin{remk}\label{orbsigmaI}Given a super-modular fusion category $\C$, let $\varphi$ be a Galois conjugate of homomorphism of $\text{dim}(-)$, that is, there exists a $\sigma\in \text{Gal}(\C)$ such that $\varphi(X)=\sigma(\text{dim}(X))$ for all $X\in\Q(\C)$. Then $\varphi$ is determined by the object $\hat{\sigma}(I)\in\Pi_0$. Indeed, for any $X\in \Pi_0$,
\begin{align}
\varphi(X)=\sigma(\text{dim}(X))=\sigma(\frac{s_{I,X}}{s_{I,I}})
=\frac{s_{\hat{\sigma}(I),X}}{s_{\hat{\sigma}(I),I}}=h_{\hat{\sigma}(I)}(X),
\end{align}
meanwhile for object $X\in\Pi_1$, there exists a unique object $Y\in\Pi_0$ such that $X=\chi\otimes Y$, and
 \begin{align}
\varphi(X)=\sigma(\text{dim}(X))=\sigma(\text{dim}(Y))
=\frac{s_{\hat{\sigma}(I),Y}}{s_{\hat{\sigma}(I),I}}
=\frac{s_{\hat{\sigma}(I),X}}{s_{\hat{\sigma}(I),I}}=h_{\hat{\sigma}(I)}(X).
\end{align}
Hence, a super-modular fusion category $\C$ is integral if and only if $\text{dim}(X)\in\Z$ for all objects $X\in\Q(\C)$ \cite[Exercise 9.6.2]{EGNO}, if and only if $\hat{\sigma}(I)=I$ for all $\sigma\in \text{Gal}(\C)$. If $\text{dim}(\C)\in\Z$, then $\text{dim}(\hat{\sigma}(I))^2=1$ for all $\sigma\in\text{Gal}(\C)$.
\end{remk}

\subsection{Formal codegrees and $d$-numbers}\label{subsection2.3}
Given   a fusion category $\C$, let $\text{Irr}(\text{Gr}(\C))$ be the set of isomorphism classes of irreducible representations of   $\text{Gr}(\C)$ over $\FC$. For   irreducible representations $\varphi,\varphi'\in \text{Irr}(\text{Gr}(\C))$, let $\text{Tr}_\varphi(-)$ be the ordinary trace function on representation $\varphi$, then   there exists a central  element
\begin{align}
\alpha_\varphi:=\sum_{X\in\Q(\C)}\text{Tr}_\varphi(X)X^*\in \text{Gr}(\C)\otimes_\Z\FC
 \end{align}
 such that $\varphi'(\alpha_\varphi)=0$ if $\varphi\ncong\varphi'$, and $f_\varphi=\varphi(\alpha_\varphi)$ is a   positive algebraic integer \cite{Lusztig,O1}, $f_\varphi$ is called a formal codegree of $\C$ \cite{O1}.

Moreover, it was proved in \cite[Corollary 2.14]{O2} that $\frac{\text{dim}(\C)}{f_\varphi}$ is also an algebraic integers. And formal codegrees of fusion category $\C$ satisfy the following equation \cite[Proposition 2.10]{O2}
\begin{align}\label{classequation}
\sum_{\varphi\in \text{Irr}(\text{Gr}(\C))}\frac{\varphi(1)}{f_\varphi}=1.
\end{align}
Obviously, we have $f_\varphi\geq1$. Indeed, if the spherical fusion category $\C$ is non-trivial, then $f_\varphi>\sqrt{\frac{2\text{rank}(\C)}{\text{rank}(\C)+1}}\geq\sqrt{\frac{4}{3}}$ for all $\varphi\in \text{Irr}(\text{Gr}(\C))$
by \cite[Theorem 4.2.1]{O3}.

Assume that there exists a  ring homomorphism $\phi$ from $\text{Gr}(\C)$ to $\FC$, then the corresponding  formal codegrees $f_\phi$   is given by
\begin{align}
\sum_{X\in\Q(\C)}\phi(X) \phi(X^*)=\sum_{X\in\Q(\C)}\phi(XX^*)=\sum_{X\in\Q(\C)}|\phi(X)|^2.
 \end{align}
In particular,   $\text{FPdim}(\C)$ and its Galois conjugates are formal codegrees of $\C$. Similarly, if fusion category $\C$ is pivotal,  then $\text{dim}(\C)$ is also a formal codegree.
Moreover, if $\C$ is a modular fusion category, then   formal codegrees of $\C$ are $\frac{\text{dim}(\C)}{\text{dim}(X)^2}$ for objects $X\in\Q(\C)$. Indeed, since $s_{I,X}=\text{dim}(X)$, the Verlinde formula \cite{BK,EGNO} implies that
\begin{align}
\sum_{Y\in\Q(\C)}h_X(Y)h_X(Y^*)=\sum_{Y\in\Q(\C)}\frac{s_{X,Y}}{\text{dim}(X)}\frac{s_{X,Y^*}}{\text{dim}(X)}
=\frac{\text{dim}(\C)}{\text{dim}(X)^2}.
\end{align}

An algebraic integer $\alpha$ is a $d$-number \cite[Definition 1.1]{O1}, if in the algebraic integer ring, the ideal generated by $\alpha$ is invariant under the action of Galois group $\text{Gal}(\overline{\mathbb{Q}}/\mathbb{Q})$. Equivalently, there exists a nonzero polynomial $f(x)=x^n+a_1x^{n-1}+\cdots a_{n-1}x+a_n\in\Z[x]$  such that $f(\alpha)=0$ and $(a_n)^i|(a_i)^n$  for all $1\leq i\leq n$, see \cite[Lemma 2.7]{O1} for more equivalent conditions. In a pre-modular fusion category $\C$, $\text{dim}(X)$ of simple objects $X$ are   $d$-numbers \cite[Theorem 1.8]{O1}. Moreover, it was proved in \cite[Theorem 1.2]{O1} that formal codegrees of  fusion categories are $d$-numbers. However, formal codegrees of  fusion rings  are not $d$-numbers in general, see \cite[Example 1.6]{O1}. Thus, we can use this property to detect whether a fusion ring is categorifiable, this is called the $d$-number test.

In addition,  we say that an algebraic integer $\alpha$ is totally real/positive, if $\alpha$ is still real/positive under any embedding of algebraic integers into field $\FC$. For example, $\sqrt{5}$ is not a totally positive integer. For any fusion category $\C$, let $X\in\Q(\C)$, $\text{FPdim}(X)^2$ is totally positive, thus $\text{FPdim}(\C)$ is totally positive. Indeed, all formal codegrees of $\C$ are totally positive by \cite[Remark 2.12]{O2}.
We also use the following theorem \cite[Corollary 8.53]{ENO1}, which is called   cyclotomic test  in \cite{O2}.
\begin{theo}\label{cyclotomic}Given a fusion category $\C$, let $\chi_1 $ be an irreducible representation of Grothendieck ring $\text{Gr}(\C)$. Then $\chi_1$ is defined over $\mathbb{Q}(\xi)$ for some root of unity $\xi$.
\end{theo}
Therefore, $\text{FPdim}(X),\text{FPdim}(\C)\in\Z[\xi]$, $\forall X\in\Q(\C)$. That is, the Galois groups of minimal polynomials defining the FP-dimensions of simple objects and $\C$ have to be abelian. Hence, in this paper,  we  use program  GAP   to distinguish whether    the Galois groups of minimal polynomials of $\text{FPdim}(X)$ are abelian, also we can use GAP to do the $d$-number test.

\section{Lagrange theorem  for dimension fusion categories}\label{section3}
In this section, we   prove the Lagrange theorem for global dimension of fusion categories.
 Given a braided fusion category $\C$, an algebra $A\in\C$ is connected, if $\text{dim}_{\FC}(\text{Hom}_\C(I,A))=1$;   a commutative  algebra $A$ is said to be an \'{e}tale algebra  if $\C_A$ is semisimple
 \cite[Proposition 2.7]{DMNO}, where $\C_A$ is the category of right $A$-modules in $\C$.
Given a connected \'{e}tale algebra $A$ in $\C$,  the  subcategory $ \C^0_A\subseteq\C_A$   of dyslectic (or local) modules of $A$ in $ \C$ is   a braided fusion category.  See \cite{DMNO,KiO} for details about \'{e}tale algebras and their dyslectic modules.

Let $\C$ be a non-degenerate fusion category,   recall that a connected \'{e}tale algebra $A\in\C$ is a Lagrangian algebra if $\text{FPdim}(\C)=\text{FPdim}(A)^2$
\cite[Definition 4.6]{DMNO}. Now we are ready to   give a proof of  the Lagrange theorem for global dimension; for pseudo-unitary fusion categories, this is  exactly \cite[Proposition 8.15]{ENO1}.
\begin{theo}\label{Lagrange}Assume that $\C$ is a fusion category. If $\D\subseteq\C$ is a fusion subcategory, then $\frac{\text{dim}(\C)}{\text{dim}(\D)}$ is an algebraic integer.
\end{theo}
\begin{proof}We can assume   $\C$  to be  a spherical fusion category. In fact,  it follows from \cite[Remark 3.1]{ENO1} and \cite[Proposition 5.14]{ENO1} that   the pivotalization $\widetilde{\C}$ of $\C$ is a spherical fusion category, and    $\text{dim}(\widetilde{\C})=2\text{dim}(\C)$
\cite[Remark 7.21.11]{EGNO}. Consequently, $\frac{\text{dim}(\C)}{\text{dim}(\D)}=\frac{\text{dim}(\widetilde{\C})}{\text{dim}(\widetilde{\D})}$, so we can replace $\C,\D$ by spherical fusion categories $\widetilde{\C}$ and $\widetilde{\D}$, respectively.

For any fusion subcategory $\D\subseteq\C$, up to isomorphism,  \cite[Theorem 4.10]{DMNO} says that there exists  a unique connected \'{e}tale subalgebra $A\subseteq\Phi(I)$ corresponding to $\D$, which satisfies equation
 \begin{align}
\frac{\text{FPdim}(\C)}{\text{FPdim}(\D)}=\text{FPdim}(A),
 \end{align}
 where $\Phi$ is the right adjoint functor of the forgetful tensor functor $F:\Y(\C)\to\C$, and  $\Phi(I)$ is  a Lagrangian algebra of $\Y(\C)$  \cite[Proposition 4.8]{DMNO}. Notice that  \cite[Theorem 4.10]{DMNO} also shows that there exists   a braided tensor equivalence of non-degenerate fusion categories $\Y(\D)\cong\Y(\C)^0_A$, we show that it is a modular equivalence below.

By definition,  $A=\oplus_{X\in\Q(\Y(\C))}[A:X]X$, where
\begin{align}
[A:X]:=\text{dim}_{\FC}(\text{Hom}_{\Y(\C)}(A,X)),
\end{align}
hence  we get that
\begin{align}
\text{dim}(A)=\sum_{X\in\Q(\Y(\C))}[A:X]\text{dim}(X).
\end{align}
If $[A:X]$ is non-zero, then   $\text{dim}(X)=\frac{\text{dim}(\C)}{f_\chi}$ for some formal codegree $f_\chi$ of $\C$ \cite[Theorem 2.13]{O2}.  Since  formal codegrees of $\C$ are positive algebraic integers,   $\text{dim}(A)>0$. Thus,   \cite[Remark 3.4]{DMNO} shows that the \'{e}tale algebra $A$ is a rigid $\C$-algebra in sense of \cite[Definition 1.11]{KiO}.
Since $\C$ is a spherical fusion category, Drinfeld center $\Y(\C)$ is a modular fusion category by \cite[Theorem 6.4]{Mu2}.  Meanwhile,  $\theta_{\Phi(I)}=\text{id}_{\Phi(I)}$ by \cite[Theorem 4.1]{NS} or \cite[Theorem 2.7]{O3}, where $\theta$ is the ribbon structure of $\Y(\C)$, then $\theta_A=\text{id}_A$ as $A\subseteq\Phi(I)$, and \cite[Theorem 4.5]{KiO} says that $\Y(\D)\cong\Y(\C)^0_A$ as modular fusion categories  as desired.

As $\text{dim}(\Y(\D))=\text{dim}(\D)^2$ by \cite[Theorem 2.15]{ENO1} \cite[Theorem 1.2]{Mu2},  again we deduce from \cite[Theorem 4.5]{KiO} that
\begin{align}
\text{dim}(\D)^2=\text{dim}(\Y(\D))=\text{dim}(\Y(\C)^0_A)
=\frac{\text{dim}(\Y(\C))}{\text{dim}(A)^2}=\frac{\text{dim}(\C)^2}{\text{dim}(A)^2},
\end{align}
so the ratio $\frac{\text{dim}(\C)}{\text{dim}(\D)}=\text{dim}(A)$ is an algebraic integer.
\end{proof}
\begin{remk}\label{nonequratios}Let $\C$ be a fusion category. For any  fusion subcategory $\D\subseteq\C$,  Theorem \ref{Lagrange} also says  that    $\frac{\text{dim}(\C)}{\text{dim}(\D)}=\frac{\text{FPdim}(\C)}{\text{FPdim}(\D)}$ if and only if $\text{FPdim}(A)=\text{dim}(A)$. This  equality fails in general, however. For example, let  $\D=YL$ be the Yang-Lee fusion category of global dimension $\frac{5+\sqrt{5}}{2}$,   its conjugate  $\overline{YL}$ has   global dimension $\frac{5-\sqrt{5}}{2}$. Let $\C=YL\boxtimes \overline{YL}$, so $\text{dim}(\C)=5$ and $\text{FPdim}(\C)=(\frac{5+\sqrt{5}}{2})^2=\frac{15+5\sqrt{5}}{2}$. The  two ratios above are not equal obviously.
\end{remk}
Given two fusion categories $\C$ and $\D$, recall that a tensor functor $F:\C\to\D$ is said to be surjective, if every simple object of $\D$ is a subobject of $F(X)$ for some object $X\in\C$; $F$ is injective if $F$ is bijective on sets of morphisms \cite[Definition 1.8.3]{EGNO}.

Same as \cite[Corollary 8.11]{ENO1}, we have the following corollary:
\begin{coro}\label{quotient}Let $\C$ and  $\D$ be fusion categories, assume that $F:\C\to\D$ is a surjective tensor functor. Then $\frac{\text{dim}(\C)}{\text{dim}(\D)}$ is an algebraic integer.
\end{coro}
\begin{proof}Since $F$ is a surjective tensor functor,  we can regard $\D$ as an indecomposable  left $\C$-module category via tensor functor $F$. Let $\C^*_\D$ and $\D_\D^*$ be the    tensor categories of left module functors of $\C$ and $\D$ with respect to module category $\D$, respectively. This is the so-called  exact pair $(F,\D)$, see \cite[$\S 7.17$]{EGNO} and \cite{ENO1} for details.

Therefore, we deduce from \cite[Corollary 7.12.13]{EGNO} that both $\C^*_\D$ and $\D_\D^*$ are fusion categories. Then we  obtain an injective tensor functor $F^*:\D^*_\D\to\C^*_\D$ by \cite[Proposition 5.3]{ENO1}, hence we can identify $\D^*_\D$ as a fusion subcategory of $\C^*_\D$.
It follows from \cite[Proposition 9.3.9]{EGNO} that we have  the following equation
\begin{align}
\frac{\text{dim}(\C)}{\text{dim}(\D)}= \frac{\text{dim}(\C^*_\D)}{\text{dim}(\D_\D^*)},
\end{align}
 which is an algebraic integer by Theorem  \ref{Lagrange}. This completes the proof.
\end{proof}
In Corollary \ref{quotient}, it is easy to see that $\frac{\text{dim}(\C)}{\text{dim}(\D)} \neq\frac{\text{FPdim}(\C)}{\text{FPdim}(\D)}$ in general. For example, let $\D$ be a  (spherical) fusion category such that $\text{dim}(\D)\neq \text{FPdim}(\D)$, and $\C:=\Y(\D)$. Let $F:\Y(\D)\to\D$ be the forgetful tensor functor, so $F$ is surjective.  Then \cite[Theorem 2.15, Proposition 8.12]{ENO1} say that
$\frac{\text{dim}(\C)}{\text{dim}(\D)}=\text{dim}(\D)\neq \text{FPdim}(\D)=\frac{\text{FPdim}(\C)}{\text{FPdim}(\D)}$.

\section{Classification of pre-modular fusion categories}
\subsection{Pre-modular fusion categories of global dimension $7$}\label{subsection4.1}
In this subsection, we show that pre-modular fusion categories of global dimension $7$ are pointed. Spherical fusion categories $\C$ of integer dimension   with $\text{dim}(\C)\leq 5$ were classified completely in \cite[Example 5.1.2]{O3}. Explicitly, $\C$ is pointed,  or $\C$ is tensor equivalent to an Ising category $\I$, or $\C$ is tensor equivalent to the Deligne tensor product   $YL\boxtimes \overline{YL}$.

We first classify pre-modular fusion categories of global dimension $6$.
\begin{prop}\label{predimen6}Let   $\C$ be a pre-modular fusion category of global dimension $6$, then $\C$ is   integral. Thus, $\C$ is either pointed or $\C\cong\text{Rep}(S_3)$, where $S_3$ is the symmetric group of order $6$.
\end{prop}
\begin{proof}\cite[Remark 4.2.3]{O3} shows that $\text{rank}(\C)=6$ if and only if $\C$ is pointed.  Obviously, $\C$ is integral if $\C'=\C$. If $\C'$ is a proper symmetric subcategory of $\C$, then $\C'$ is an integral fusion category 
\cite[Corollary 9.9.1]{EGNO}, so $\text{dim}(\C')=1,2,3$ by Theorem \ref{Lagrange}. If $\C'=\text{Rep}(G)$ is Tannakian, then $\C\cong\D^G$, where $\D$ is a modular fusion category with $\dim(\D)=\frac{\dim(\C)}{|G|}$, so $\D\cong\C(\Z_2,\eta_1)$ or $\D\cong\C(\Z_3,\eta_2)$  by \cite[Example 5.1.2]{O3}, thus $\C\cong\text{Rep}(S_3)$ or $\C$ is pointed. If  $\C'\cong \text{sVec}$, then $\text{rank}(\C)=4,6$.  If $\text{rank}(\C)=4$,
 for any simple object $X\in\C$, $\text{dim}(X)^2\in{\{1,2}\}$ by equation \ref{halfdimen}, while  \cite[Corollary 3.4]{Yu} shows $\frac{6}{2\text{dim}(X)^2}$ is an algebraic integer, this is impossible. If $\C$ is modular, conclusions of \cite{BNRW,RSW} imply that there does not exist modular fusion category  with $\text{rank}(\C)<6$ and $\text{dim}(\C)=6$. Therefore,  $\C$ must be an integral fusion category.
\end{proof}

Assume  that $p$ is a prime, and let $\C$ be a pre-modular fusion category of global dimension $p$. Since $\C'$ is a symmetric fusion subcategory of $\C$,  $\text{dim}(\C')=\text{FPdim}(\C')$ is an integer, Theorem \ref{Lagrange} shows that rational number $\frac{p}{\text{dim}(\C')}$ is an algebraic integer. Then $\C$ is symmetric if  $p=\text{dim}(\C')$, otherwise $\C$ is modular. If $\C$ is symmetric, then $\C\cong \text{Rep}(\Z_p)$ or $\C\cong \text{sVec}$ as symmetric fusion category, so $\C$ is a pointed fusion category.

Assume $\C$ is  modular below.  If $p>5$, it follows from \cite{BNRW,RSW} that there do  not exist   modular  fusion categories of global dimension $p$ with rank  less than $6$, so $6\leq \text{rank}(\C)\leq p$ and $\C\cong\C(\Z_p,\eta)$ if $\text{rank}(\C)=p$ by \cite[Lemma 4.2.2, Remark 4.2.3]{O3}.  Let  $\D\subseteq\C$ be an arbitrary fusion subcategory,  then $\D$ is  a modular fusion category, since $\D\cap\D_\C'$ is symmetric and $\text{dim}(\D\cap\D_\C')^2$ divides $p$ by Theorem \ref{Lagrange}. Thus  \cite[Theorem 3.13]{DrGNO2} says that $\C\cong\A_1\boxtimes\cdots\boxtimes\A_s$, where $\A_i$ are simple modular fusion subcategories of $\C$, $1\leq i\leq s<\text{rank}(\C)\leq p$.

\begin{theo}\label{moddimen7}Let $\C$ be a pre-modular fusion category of global dimension $7$. Then $\C$ is pointed. That is, $\C\cong\C(\Z_7,\eta)$ or $\C\cong\text{Rep}(\Z_7)$.
\end{theo}
\begin{proof}If $\C'$ is non-trivial, then  $p=\text{dim}(\C')$ by Theorem \ref{Lagrange} and $\C\cong\text{Rep}(\Z_7)$. Assume that $\C$ is modular and that $\C$ is not pointed below, in particular, $\C$ is not pseudo-unitary, then previously argument says  $\C$ is a simple modular fusion category of rank $6$. Meanwhile, Cauchy Theorem \cite{BNRW,DongLinNg} says $\FQ(S_\C)\subseteq\FQ(T_\C)=\FQ(\zeta_{7^n})$ for some integer $n\geq1$, where $\zeta_{7^n}$ is a primitive $7^n$-th root of unity, we know
\begin{align}
\text{Gal}(\FQ(T_\C)/\FQ)\cong\Z_2\times\Z_3\times\Z_{7^{n-1}}.
  \end{align}
Since homomorphisms  $\text{dim}(-)$ and $\text{FPdim}(-)$ take values in the maximal totally real subfield $\FQ(\zeta_{7^n}+\zeta^{-1}_{7^n})$, they  are not in the same orbits under the action of the Galois group \begin{align}
\text{Gal}(\FQ(\zeta_{7^n}+\zeta^{-1}_{7^n})/\FQ)\cong\Z_3\times\Z_{7^{n-1}},
 \end{align}
for $\C$ is not a pseudo-unitary fusion category. Meanwhile,  $\text{rank}(\C)=6$,  so each orbit of $\text{dim}(-)$ and $\text{FPdim}(-)$ has  exactly three homomorphisms.

Let $f_1\leq f_2\leq f_3=\text{FPdim}(\C)$ be  distinct Galois conjugates of $\text{FPdim}(\C)$, so formal codegrees of $\C$ are $7,7,7,f_1,f_2,f_3$. By equation \ref{classequation}, we have
\begin{align}
\frac{3}{7}+\frac{1}{f_1}+\frac{1}{f_2}+\frac{1}{f_3}=1,
\end{align}
so, $f_1>\frac{7}{4}$, $f_2>\frac{7}{2}$ and $f_3>7$. Again, $f_1f_2f_3|7^3$, so $f_1f_2f_3=49$ or  $f_1f_2f_3=343$. If $f_1f_2f_3=49$, then $d$-number test \cite[Lemma 2.7]{O1} shows that $f_1, f_2,f_3$ are roots of equation $x^3-14x^2+28x-49=0$, which fails to satisfy cyclotomic test in Theorem \ref{cyclotomic}; and similarly, if  $f_1f_2f_3=343$, then  $f_1, f_2,f_3$ are roots of equation $x^3-\alpha x^2+196x-343=0$, where $\alpha=f_1+f_2+f_3$. Obviously, $d$-number test \cite[Lemma 2.7]{O1} shows that $7|\alpha$, meanwhile
\begin{align}
\alpha<(f_1-\frac{3}{4})f_2+(f_3-6)f_1+(f_2-\frac{5}{2})f_3< 196-\frac{3}{2}\alpha,
\end{align}
then $14\leq\alpha\leq78$. However,  a direct computation shows that  there is no solution for cyclotomic test in Theorem \ref{cyclotomic} if $\alpha\neq 28$. While when $\alpha=28$, roots of equation $x^3-28 x^2+196x-343=0$ are equal to $7,\frac{21\pm7\sqrt{5}}{2}$, this is impossible as $\sqrt{5}\notin\FQ(\zeta_{7^n}+\zeta^{-1}_{7^n})$.

In summary, pre-modular fusion categories of global dimension $7$ are pointed.
\end{proof}

\begin{ques}Let $p$ be a prime, is pre-modular fusion category of  dimension $p$   pointed or  braided tensor equivalent to Deligne tensor product $YL\boxtimes \overline{YL}$\footnote{Recently, a positive answer was given by \cite{Sch}.}?
\end{ques}
\subsection{Pre-modular fusion categories of global dimensions $8$, $9$ and $10$}\label{subsection4.2}
In this subsection,   we give a complete classification of pre-modular fusion categories of global dimensions $8$, $9$ and $10$ respectively.

Weakly integral braided fusion categories of dimension $8$ were classified in \cite[Remark 4.7]{Yu}.
\begin{lemm}\label{superdim8}Let $\C$  be a super-modular fusion category of global dimension $8$, then $\C$ is  weakly integral. Thus $\C\cong\text{sVec}\boxtimes\D$ as super-modular fusion category, where $\D$ is either pointed or $\D\cong\I$ is an Ising category.
\end{lemm}
\begin{proof}Let $\C$ be a super-modular fusion category of global dimension $8$, then \cite[Lemma 3.28]{DrGNO2} implies that $\text{rank}(\C)$ is even, and $\text{rank}(\C)\leq8$ by \cite[Lemma 4.2.2]{O3}. Assume that $\Pi_0={\{I,X_1,\cdots,X_n}\}$, hence $n=2,3$. In addition, $4=1+d_{X_1}^2+\cdots +d_{X_n}^2$, where $d_{X_i}:=\text{dim}(X_i)$  for any $1\leq i\leq n$, and   $d_{X_i}^2$ divides $4$ by
\cite[Corollary 3.4]{Yu}. We deduce from
\cite[Remark 4.2.3]{O3} that $\C$ is pointed if $n=3$.

If $n=2$, we consider the induced action of $\text{Gal}(\C)$ on $\Pi_0$,   $d_{\hat{\sigma}(I)}^2=1$ by Remark \ref{orbsigmaI}. If $\hat{\sigma}(I)=I$ for all $\sigma$, then $\C$ is integral by Remark \ref{orbsigmaI}. However, there does not exist non-pointed integral super-modular fusion categories which have FP-dimension $8$ \cite[Remark 4.7]{Yu}.

Hence there exists a $\sigma\in \text{Gal}(\C)$ such that $\hat{\sigma}(I)\ncong I$, assume that $\hat{\sigma}(I)=X_1$ and $d_{X_2}^2=2$. Then it is easy to see that
\begin{align*}
\hat{S}=\left(
          \begin{array}{ccc}
            1 & 1 & d_{X_2} \\
            1 & 1 & -d_{X_2} \\
           d_{X_2} & -d_{X_2} & 0 \\
          \end{array}
        \right),
\end{align*}
we can obtain the naive fusion rules from matrix $\hat{S}$   \cite[Proposition 2.7]{BGNPRW}. Moreover it is easy to see that $\C$ contains  a fusion subcategory with fusion rules as Ising category $\I$, see also \cite[Lemma 3.8]{BGNPRW}. Thus $\C\cong\I\boxtimes\text{sVec}$ as $\I$ is non-degenerate \cite[Corollary B.12]{DrGNO2}, $\C$ is weakly integral.
\end{proof}
\begin{remk}\label{totallypost}Let $\C$ be a fusion category such that $\text{FPdim}(-)$ takes values in $\FQ(\sqrt{n})$, where $n$ is a non-negative integer. Let $\Q(\C)={\{X_i|i=0,\cdots,t}\}$ with $X_0:=I$, $\text{FPdim}(X_i)^2=\frac{a_i+b_i\sqrt{n}}{2}$, where $a_i,b_i$ are  rational numbers, $1\leq\forall i\leq t$.
\begin{align}
\text{FPdim}(\C)=\sum_{i=0}^t\text{FPdim}(X_i)^2=1+\sum^t_{i=1}\frac{a_i+b_i\sqrt{n}}{2},
 \end{align}
  If $\text{FPidm}(\C)\notin\Z$, there exists an $i_0$ such that $a_{i_0}$ and $b_{i_0}$ are positive integers. Indeed,  by assumption   $b_{i_0}\neq0$ for some $i_0$.  If $b_{i_0}<0$, then $\frac{a_{i_0}-b_{i_0}\sqrt{n}}{2}>\frac{a_{i_0}+b_{i_0}\sqrt{n}}{2}$, this contradicts to property of FP-dimension of simple object \cite[Proposition 3.3.4]{EGNO}. If  $a_{i_0}\leq0$, then   $\frac{a_{i_0}+b_{i_0}\sqrt{n}}{2}$ can not  a totally positive algebraic integer.

Next, we show $a_{i_0}$ and $b_{i_0}$ are integers. Let $\alpha_{i_0}:=\frac{a_{i_0}+b_{i_0}\sqrt{n}}{2}$, then algebraic integer $\alpha_{i_0}$ is a root  of equation  $4x^2-4a_{i_0}x+a_{i_0}^2-nb_{i_0}^2=0$. Hence, $a_{i_0}\in\Z$. Since $\frac{a_{i_0}+b_{i_0}\sqrt{n}}{2}$ is a totally positive algebraic integer, $\frac{a_{i_0}-b_{i_0}\sqrt{n}}{2}$ is also positive. Let $4m=a_{i_0}^2-nb_{i_0}^2$, where $m$ is a positive integer, then $b_{i_0}=\sqrt{\frac{a_{i_0}^2-4m}{n}} $ is a rational number, which means $b_{i_0}$ is an integer.
\end{remk}
Let  $\C$ be a modular fusion category and denote $\text{Gal}(\C):=\text{Gal}(\FQ(S_\C)/\FQ)$. Recall that  given an arbitrary $\sigma\in\text{Gal}(\C)$, there exists a unique object $\hat{\sigma}(X)\in\Q(\C)$ such that
\begin{align}
\sigma(s_{X,Y}/\text{dim}(X))=s_{\hat{\sigma}(X),Y}/\text{dim}(\hat{\sigma}(X)), ~\forall X,Y\in\Q(\C).
\end{align}
\begin{theo}\label{moddimen8}Let  $\C$ be a modular fusion category of global dimension $8$, then $\C$ is not simple.
\end{theo}
 \begin{proof}On the contrary, assume that $\C$ is simple,  thus  $\C_\text{int}=\text{Vec}$ and $\text{FPdim}(\C)\notin\Z$ by \cite[Remark 4.7]{Yu}. In addition,  $\text{rank}(\C)=6,7$ \cite{BNRW,O3}. We only give a proof  for $\text{rank}(\C)=6$, the case when $\text{rank}(\C)=7$ can be proved similarly. By \cite[Corollary 8.18.2]{EGNO}  $\text{ord}(T_\C)$ divides $8^{\frac{5}{2}}$, so  $\FQ(S_\C)\subseteq\FQ(T_\C)=\FQ(\zeta_{2^n})$ for some $ n\leq7$ \cite{DongLinNg}. While $\FQ(\zeta_2)=\FQ$, $\FQ(\zeta_4)=\FQ(\text{i})$ where $\text{i}^2=-1$, and $\text{FPdim}(\C)$ is a totally positive algebraic integer, thus $n\geq3$. Note that $\text{Gal}(\FQ(\zeta_{2^n})/\FQ)=\Z_2\times\Z_{2^{n-2}}$, hence the number of Galois conjugates for any homomorphism $\phi:\text{Gr}(\C)\to\FC$ is power of $2$. Let $\Gamma:={\{\hat{\sigma}(I)\in\Q(\C)|\sigma\in \text{Gal}(\C)}\}$.
  Since $\C$ is not weakly integral, the number of Galois conjugates of $\text{dim}(-)$ are $2$ or $4$, equivalently $|\Gamma|=2,4$.

Case (1): If  $\text{dim}(-)$ and $\text{FPdim}(-)$ each have exactly $2$ Galois conjugates, so they take values in $\FQ(\sqrt{2})$. Let $f_1,f_2=\text{FPdim}(\C)$ be conjugated formal codegrees of $\C$, then $16\leq f_1f_2\leq64$ and $f_1f_2|64$. Since $f_1$ and $f_2$ are roots of $x^2-\beta x+f_1f_2=0$, $d$-number test says that $f_1f_2|\beta^2$ \cite[Proposition 2.7]{O1}. It is easy to show that  $\FQ(f_2)=\FQ(\sqrt{2})$  if and only if for equation $x^2-16x+32=0$, in this case $f_1=8-4\sqrt{2}$ and $\text{FPdim}(\C)=8+4\sqrt{2}$. Thus $\C$ contains a simple object $X$ such that $\dim(X)^2=\frac{\dim(\C)}{\text{FPdim}(\C)}=2-\sqrt{2}$,   it is easy to see $\dim(X)\notin\FQ(\sqrt{2})$, which is contradicting to the assumption  that homomorphism $\text{dim}(-)$   takes values in $\FQ(\sqrt{2})$.

Case (2): If  $\text{dim}(-)$ has $4$ Galois conjugates, then $\Q(\C)=\Gamma\cup{\{X_4,X_5}\}$, where $\Gamma={\{I,X_1,X_2,X_3}\}$, homomorphism $\text{FPdim}(-)$ and its Galois conjugate are determined by objects $X_4$ and $X_5$, respectively.
Let $\sigma\in\text{Gal}(\C)$ such that $\hat{\sigma}(I)=X_1$, then $\hat{\sigma}(X_4)=X_5$ or $\hat{\sigma}(X_4)=X_4$.  By definition $\sigma(d_Y)=\sigma(s_{I,Y}/d_I)=s_{X_1,Y}/d_{X_1}$ for all $Y\in\Q(\C)$, and  $d_{X_i}^2=1$ for $1\leq i\leq3$. Hence, $s_{X_1,Y}=\sigma(d_Y)d_{X_1}$. Orthogonality of $s_{I,-}$ and $s_{X_1,-}$ shows that
\begin{align}
4+d_{X_4}\sigma(d_{X_4})+d_{X_5}\sigma(d_{X_5})=0;
\end{align}
 and $4+d_{X_4}^2+d_{X_5}^2=8=\text{dim}(\C)$. If $\hat{\sigma}(X_4)=X_4$ and $\hat{\sigma}(X_5)=X_5$, then \begin{align}
\frac{8}{\sigma(d_{X_4}^2)}
=\sum_{Y\in\Q(\C)}\sigma(\frac{s_{X_4,Y}s_{X^*_{4},Y}}{d^2_{X_4}})
=\sum_{Y\in\Q(\C)}\frac{s_{X_4,Y}s_{X^*_{4},Y}}{d^2_{X_4}}=\frac{8}{d_{X_4}^2},
 \end{align}
 hence $\sigma(d_{X_4}^2)=d_{X_4}^2$ and $\sigma(d_{X_5}^2)=d_{X_5}^2$.

 Notice that $\frac{8}{d_{X_4}^2}$ and $\frac{8}{d_{X_5}^2}$ are conjugated formal codegrees of $\C$, so $ \frac{64}{d_{X_4}^2d_{X_5}^2} \in\Z$. Obviously, $\sigma(d_{X_4})=d_{X_4}$ and $\sigma(d_{X_5})=d_{X_5}$ can not be true at the same time; and if $\sigma(d_{X_4})=-d_{X_4}$ and $\sigma(d_{X_5})=d_{X_5}$, then we obtain $d_{X_5}^2=0$, impossible. Hence, $\sigma(d_{X_4})=-d_{X_4}$ and $\sigma(d_{X_5})=-d_{X_5}$, thus  $4=d_{X_4}^2+d_{X_5}^2$, $d_{X_4}^2$ and $d_{X_5}^2$ are roots of $x^2-4x+d_{X_4}^2d_{X_5}^2=0$, then $d_{X_4}^2d_{X_5}^2\leq4$.
 If $d_{X_4}^2d_{X_5}^2=4$, then $d_{X_4}^2=d_{X_5}^2=2$ and $\text{FPdim}(\C)=4$,  contradiction; if $d_{X_4}^2d_{X_5}^2=1$, then $d_{X_4}^2=2\pm\sqrt{3}\notin\FQ(\zeta_{2^n})$,  impossible. So  $d_{X_4}^2d_{X_5}^2=2$, consequently $\text{FPdim}(\C)=\frac{8}{2-\sqrt{2}}=8+4\sqrt{2}$, previous argument shows that it is impossible.
If $\hat{\sigma}(X_4)=X_5$ and $\hat{\sigma}(X_5)=X_4$, then $\sigma(d_{X_4}d_{X_5})=d_{X_4}d_{X_5}$   as $4+d_{X_4}\sigma(d_{X_4})+d_{X_5}\sigma(d_{X_5})=0$. If $\sigma(d_{X_4})=d_{X_5}$ and $\sigma(d_{X_5})=d_{X_4}$, then $d_{X_4}=-d_{X_5}$ and $d_{X_4}^2=2$; if $\sigma(d_{X_4})=-d_{X_5}$ and $\sigma(d_{X_5})=-d_{X_4}$, then $d_{X_4}=d_{X_5}$ and $d_{X_4}^2=2$. From both cases we obtain that $\text{FPdim}(\C)=4$, it is a contradiction.

Case (3): If homomorphism $\text{FPdim}(-)$ has 4 Galois conjugates and homomorphism $\text{dim}(-)$ has 2 Galois conjugates, then $\text{dim}(-)$ takes values in $\FQ(\sqrt{2})$, since $\FQ(\zeta_{2^7}+\zeta^{-1}_{2^7})$ has a unique quadratic subfield  $\FQ(\sqrt{2})$, so  $\text{FPdim}(\C)=\frac{10}{\text{dim}(Z_0)^2}\in\FQ(\sqrt{2})$ for some simple object $Z_0\in\C$. Let $\mathbb{K}:=\FQ(\text{FPdim}(X): \forall X\in\Q(\C))$, then \cite[Proposition 1.8]{GS} says that $\C$ is faithfully graded by the Galois group $\text{Gal}(\mathbb{K}/\FQ(\text{FPdim}(\C))$.  However,  $\C$ is assumed to be simple, the Galois group $\text{Gal}(\mathbb{K}/\FQ(\text{FPdim}(\C))$ must be trivial, thus   $\text{FPdim}(-)$ has at most two Galois conjugates, it is impossible.
In summary,  $\C$ is not a simple modular fusion category.
\end{proof}

\begin{coro}\label{predimen8}Let $\C$ be a pre-modular fusion category of global dimension $8$. Then $\C$ is   weakly integral. In particular, $\C$ is either pointed or $\C\cong\I\boxtimes\C(\Z_2,\eta)$ if $\C$ is modular.
\end{coro}
\begin{proof} If $\C$ contains a non-trivial Tannakian subcategory $\text{Rep}(G)$, then $\C\cong\D^G$ as fusion category,  $\text{dim}(\D)|G|=\text{dim}(\C)$ by \cite[Proposition 4.26]{DrGNO2}. Then $\D$ is weakly integral \cite[Example 5.1.2]{O3}, so is $\C$ by \cite[Corollary 4.27]{DrGNO2}. We  assume $\C$ doesn't contain non-trivial Tannakian subcategory below. Then it suffices to show   $\C$ is weakly integral when  $\C'=\text{Vec}$ by Lemma \ref{superdim8}.

Then it follows from  Theorem \ref{moddimen8} that $\C$ contains a non-trivial fusion subcategory  $\A$, and   $\text{dim}(\A_\C')\text{dim}(\A)=\text{dim}(\C)=8$
 by \cite[Theorem 3.10]{DrGNO2}, where  $\A_\C'$ be the centralizer of $\A$ in $\C$. Since  $\A\cap\A_\C'$ is a symmetric fusion category, it follows from Theorem  \ref{Lagrange} that $\text{dim}(\A\cap\A_\C')^2|8$, so $\text{dim}(\A\cap\A_\C')=1$ or $2$, which means $\A\cap\A_\C'\subseteq\text{sVec}$.

If   $\A\cap\A_\C'=\text{sVec}$, then $\C$ contains a super-modular fusion category $\B$ of global dimension $4$   \cite[Theorem 3.10]{DrGNO2}.  By \cite[Example 5.1.2]{O3} $\B\cong \text{sVec}\boxtimes\C(\Z_2,\eta)$. Thus $\C\cong\C(\Z_2,\eta)\boxtimes\C(\Z_2,\eta)_\C'$ as modular fusion category by \cite[Theorem 3.13]{DrGNO2}. If   $\A\cap\A_\C'=\text{Vec}$, then $\C\cong\A\boxtimes\A_\C'$
 as modular fusion category \cite[Theorem 3.13]{DrGNO2}. Assume that $\text{rank}(\A)\leq \text{rank}(\A_\C')$, hence  $\A=\C(\Z_2,\eta)$ as $\text{rank}(\C)\leq8$ \cite[Lemma 4.2.2]{O3}. So   $\C(\Z_2,\eta)_\C'$ is either braided equivalent to an Ising category $\I$ or pointed   \cite[Example 5.1.2]{O3}, which means that $\C$ is weakly integral, as desired.
\end{proof}

Next, we consider the structure of pre-modular fusion categories of global dimension $9$.
 Given a  modular fusion category $\C$, we have the Galois symmetry. More precisely,  let
 \begin{align}
 \xi(\C):=\sum_{X\in\Q(\C)}\theta_X\text{dim}(X)^2/\sqrt{\text{dim}(\C)}
  \end{align}
    be the multiplicative central charge of $\C$, where  $\sqrt{\text{dim}(\C)}$ is the positive square root of $\text{dim}(\C)$,  $\theta$ is the ribbon structure of $\C$. Then  there exists a $3$-th root $\gamma$ of $\xi(\C)$ such that
    \begin{align}\label{GaloisSym}
    \sigma^2(t_X)=t_{\hat{\sigma}(X)},~\text{ for any $\sigma\in \text{Gal}( \FQ(T_\C/\gamma)/\FQ)$}
     \end{align}
by \cite[Theorem II]{DongLinNg}, where $t_X:=\theta_X/\gamma$, $\forall X\in\Q(\C)$.
So, for a modular fusion category of given global dimension, we can use the Galois symmetric to consider orbits of homomorphisms $\text{dim}(-)$ and $\text{FPdim}(-)$, since they are determined uniquely by  simple objects  of $\C$, see subsection \ref{subsection2.2}. The proof of the following theorem is inspired by \cite{Sch}.
\begin{theo}\label{dimension9}Pre-modular fusion categories of global dimension $9$ are either  pointed or braided equivalent to a Galois conjugate of modular fusion  category $\C(\mathfrak{so}_5,\zeta_{18}, 9)_\text{ad}$.
\end{theo}
\begin{proof}
Let $\C$ be an arbitrary pre-modular fusion category of global dimension $9$. Same  as Corollary \ref{predimen8}, it can be proved that $\C$ is pointed if $\C$ is not simple. Assume  $\C$ is not pointed below; in particular, $\text{FPdim}(\C)\neq9$ and  $6\leq\text{rank}(\C)\leq8$ by \cite{BNRW} and \cite[Lemma 4.2.2]{O3}.

Then we have $\FQ(S_\C)\subseteq\FQ(T_\C)=\FQ(\zeta_{3^n})$ for some primitive $3^n$-th root of unity $\zeta_{3^n}$ \cite{BNRW,DongLinNg}.   \cite[Theorem II, Lemma 2.2]{DongLinNg} say that there exists a $3$-th root $\gamma$ of $\xi(\C)$ such that $t_X=\theta_X/\gamma\in\FQ(\zeta_{3^n})$ and  $\sigma^2(t_X)=t_{\hat{\sigma}(X)}$ for all $\sigma\in\text{Gal}(\FQ(\zeta_{3^n})/\FQ)$ by equation \ref{GaloisSym}.  Let
\begin{align}
G_0:={\{\sigma|\sigma=\tau^2 ~\text{for some $\tau\in \text{Gal}( \FQ(\zeta_{3^n})/\FQ)$}}\},
 \end{align}
 hence the order of $G_0$ equals to $3^{n-1}$. For any simple object $X$ such that $\theta_X\neq1$, $\theta_X$ is a primitive   $3^k$-th root of unity for some $1\leq k\leq n$, then under the action of $G_0$, $t_X$ has $3^{k-1}$ Galois conjugates, so is $X$. Since $\text{rank}(\C)\leq8$,  we see $k\leq n\leq2$.

 If $n=k=1$, then $\text{dim}(X)^2\in\FQ(\zeta_3)$ is a $d$-number for any simple object $X$
  \cite[Theorem 1.8]{O1} , thus $\text{dim}(X)^2\in\Z$. While $\text{FPdim}(\C)=\frac{9}{\text{dim}(Z)^2}$ for some object $Z\in\Q(\C)$, so $\C$ is integral and  pointed, it is a contradiction.
 So $n=k=2$ and  $\text{Gal}(\FQ(T_\C)/\FQ)=\Z_2\times\Z_3$. Note that  both $\text{dim}(-)$ and $\text{FPdim}(-)$ take values in $\FQ(\zeta_9+\zeta_9^{-1})$, since $[\FQ(\zeta_9+\zeta_9^{-1}):\FQ]=3$ and $\C$ is not an integral fusion category, homomorphisms $\text{dim}(-)$ and $\text{FPdim}(-)$ each has exactly three Galois conjugates under the action of Galois group $\text{Gal}(\C)$.

If $\text{rank}(\C)=6$, then $\C$ is self-dual by \cite[Lemma 3.2]{BNRW}, so $s_{X,Y}$ are real algebraic integers for any simple objects $X,Y$, which means $\FQ(S_\C)\subseteq\FQ(\zeta_9+\zeta_9^{-1})$ and $\text{Gal}(\C)\cong\langle(123)(456)\rangle$. Hence, as a modular fusion category, $\C$ is  equivalent to one of the Galois conjugates of modular fusion category $\C(\mathfrak{so}_5,\zeta_{18},9)_\text{ad}$  by \cite[Theorem 3.5]{Green}.

If $\text{rank}(\C)=7$, then  there exists a  simple object $Y$ fixed by $\text{Gal}(\C)$, and the corresponding formal codegree  $\frac{9}{\text{dim}(Y)^2}$ is an integer, hence $\text{dim}(Y)^2=1$ as $\text{dim}(Y)^2\neq3$. Let $f_1$, $f_2$, $f_3=\text{FPdim}(\C)$ be conjugated formal codegrees of $\C$, formal codegrees of $\C$ are $9,9,9,9$, $f_1$, $f_2$ and $f_3$, then $\Pi_{i=1}^3f_i|729$ \cite[Proposition 9.4.2]{EGNO} and  equation \ref{classequation} says that
\begin{align*}
\frac{1}{f_1}+\frac{1}{f_2}+\frac{1}{f_3}=\frac{5}{9}.
\end{align*}
As Theorem \ref{moddimen7}, a direct computation shows that $f_i$ are roots of $x^3-\beta_1x^2+135x-243=0$ with  $9|\beta_1$ and $18\leq\beta_1\leq45$, or $x^3-\beta_1x^2+405x-729=0$ with $9|\beta_1$ and $18\leq\beta_1<150$. Integer $\beta_1$ does not satisfy the cyclotomic test (i.e., Theorem \ref{cyclotomic}) except $x^3-54x^2+405x-729=0$, however roots of  equation $x^3-54x^2+405x-729=0$ do not belong to field $\FQ(\zeta_9+\zeta_9^{-1})$, this is a contradiction.
Similarly,  the case when $\text{rank}(\C)=8$ can also be excluded by using previous arguments. This completes the proof.
\end{proof}

Let $\A:=YL\boxtimes \overline{YL}$ and $\B:=\C(\Z_5,\eta)\boxtimes\A$ be modular fusion categories, then  \begin{align*}
\text{dim}(\A\boxtimes\A)=\text{dim}(\B)=25,~\text{rank}(\A\boxtimes\A)=16,~\text{rank}(\B)=20.
\end{align*}
Moreover, there exist  modular fusion categories $\C$ of rank $3$ \cite{O2}, whose global dimensions are roots of the equation $x^3-14x^2+49x-49=0$. Assume that $\C_i$ ($1\leq i\leq3$)  are three conjugates of $\C$ such that the corresponding global dimensions are exactly conjugated roots of the previous equation. Let  $\D:=\C_1\boxtimes\C_2\boxtimes\C_3$, then $\text{dim}(\D)=49$ and   $\text{rank}(\D)=27$.
\begin{ques}Let $p>3$ be a  prime, assume that $\C$ is a  modular fusion category of global dimension $p^2$. Are there other non-pointed modular fusion categories when $p>3$?
\end{ques}

In the last, we classify  pre-modular fusion categories of global dimension $10$.
We begin with classification of super-modular fusion categories of global dimension $10$.
\begin{lemm}\label{supdimen10}Let $\C$ be a super-modular fusion category of global dimension $10$. Then  as a super-modular fusion category $\C\cong \text{sVec}\boxtimes\C(\Z_5,\eta)$  or $\C\cong \text{sVec}\boxtimes(YL\boxtimes \overline{YL})$.
\end{lemm}
\begin{proof}By using the same method as Lemma \ref{superdim8}, we know that there does not exist a rank $6$ super-modular fusion category of global dimension $10$. Thus, super-modular fusion categories $\C$ of global dimension $10$ have rank $8$ or $10$. If $\text{rank}(\C)=10$, then $\C$ is pointed
 by \cite[Remark 4.2.3]{O3}, and $\C\cong \text{sVec}\boxtimes\C(\Z_5,\eta)$ as super-modular fusion category.

 Assume $\text{rank}(\C)=8$, let $\Pi_0={\{I,X_1,X_2,X_3}\}$  and  $\Gamma:={\{\hat{\sigma}(I)|\sigma\in \text{Gal}(\C)}\}$, $d_{X_i}:=\text{dim}(X_i)$ for $1\leq i\leq3$ . If $|\Gamma|=1$, then $\C$ is integral by Remark \ref{orbsigmaI},  there is no solution for $5=1+\sum_{i=1}^3d_{X_i}^2$
  by \cite[Corollary 3.4]{Yu}, however.  If $|\Gamma|=3$, then Remark \ref{orbsigmaI} says that  there exists a simple object $Y\in\Pi_0$ such that $\text{dim}(Y)^2=2$, it is contradicting to
 \cite[Corollary 3.4]{Yu}.

Hence, $|\Gamma|=2$. Assume $\Gamma={\{I,X_1}\}$, thus $d_{X_2}^2+d_{X_3}^2=3$. We claim that   exists $\sigma\in \text{Gal}(\C)$ such that $\hat{\sigma}(X_1)=X_2$. If not, then $10/d_{X_2}^2$ is invariant under the action of group $\text{Gal}(\C)$, so $d_{X_2}^2$ and $d_{X_3}^2$ are integers, it is impossible by \cite[Corollary 3.4]{Yu}. Thus,  $d_{X_2}^2$ and $d_{X_3}^2$ are conjugated roots of the following equation $x^2-3x+m=0$ for some positive integer $m$. Since $d_{X_2}^2$ is a   $d$-number \cite[Theorem 1.8]{O1}, $m$ divides $9$ by \cite[Lemma 2.7]{O1}, so $m=1$. Without loss of generality, let   $d_{X_2}^2=\frac{3+\sqrt{5}}{2}$ and $d_{X_3}^2=\frac{3-\sqrt{5}}{2}$, so all simple objects  $X_i$ ($1\leq i\leq3$) are self-dual as they have different quantum dimensions.

Let $\sigma\in\text{Gal}(\FQ(\sqrt{5})/\FQ)\subseteq\text{Gal}(\C)$ be the unique element such that $\sigma(\sqrt{5})=-\sqrt{5}$, then
\begin{align}
\sigma(d_{X_i})=\sigma(\frac{s_{I,X_i}}{s_{I,I}})
=\frac{s_{\hat{\sigma}(I),X_i}}{d_{\hat{\sigma}(I)}}.
\end{align}
We see that $\hat{\sigma}(I)=X_1$ and $\hat{\sigma}(X_2)=X_3$, so $s_{X_1,X_i}=d_{X_1}\sigma(d_{X_i})$ for all $1\leq i\leq3$. Indeed, if $\hat{\sigma}(I)=I$, then $\sigma(d_{X_i})=d_{X_i}$, it is impossible. The orthogonality of matrix $\hat{S}$ shows that
\begin{align}
d_{X_2}d_{X_3}=-1, s_{X_2,X_3}=s_{X_2,X_2}=s_{X_3,X_3}=-1.
\end{align}
Therefore, we obtain
\begin{align*}
\hat{S}=\left(
          \begin{array}{cccc}
            1 & d_{X_1} &  d_{X_2} &  d_{X_3}\\
            d_{X_1} & 1 & d_{X_1}d_{X_3} & d_{X_1}d_{X_2} \\
            d_{X_2} &  d_{X_1}d_{X_3} & -1 & -1 \\
           d_{X_3}&  d_{X_1}d_{X_2} & -1 & -1 \\
          \end{array}
        \right).
\end{align*}

If $ d_{X_1}=1$, then   $-1=\text{dim}(X_2\otimes X_3)=a+d_{X_2}b+d_{X_3}c$ for some non-negative integers $a,b,c$, which shows that $(a+1)^2=b^2$ and $b=c$, so $b=c=1$ and $a=0$, impossible, hence $d_{X_1}=-1$. As  Lemma \ref{superdim8},   the naive fusion rules of $\C$
can be obtained through \cite[Proposition 2.7]{BGNPRW}, we see that $\C$ contains  fusion category with  fusion rules as Yang-Lee fusion category $YL$ \cite[Lemma 4.4]{BPRZ}, then $\C\cong \text{sVec}\boxtimes YL\boxtimes \overline{YL}$ as a super-modular fusion category.
\end{proof}
To classify modular fusion categories of global dimension $10$, we need the following lemma.
\begin{lemm}\label{lemdim10}Let $\C$ be a  modular fusion category of global dimension $10$. If $\text{FPdim}(X)\in\FQ(\sqrt{5})$ for any object $X\in\Q(\C)$ and $\text{FPdim}(\C)=15+5\sqrt{5}$, then $\C\cong YL\boxtimes\overline{YL}\boxtimes\C(\Z_2,\eta)$ as a modular fusion category.
\end{lemm}
\begin{proof}We only need to show $\C_\text{int}\neq\text{Vec}$. Indeed, if $\C_\text{int}\neq\text{Vec}$, then $\text{FPdim}(\C_\text{int})=\text{dim}(\C_\text{int})=2$ or $5$  by Theorem \ref{Lagrange}, obviously $\C_\text{int}$ can  not be symmetric, thus $\C_\text{int}\cong\C(\Z_2,\eta_1)$ or $\C(\Z_5,\eta_2)$ as a modular fusion category. If   $\C_\text{int}\cong\C(\Z_5,\eta_2)$, then $\C$ has to be  integral and $\text{FPdim}(\C)=10$,   it is impossible. So, $\C\cong\C(\Z_2,\eta_1)\boxtimes YL\boxtimes\overline{YL}$ by \cite[Example 5.1.2]{O3},  since $\C$ is not integral.
On the contrary, we assume that $\C_\text{int}=\text{Vec}$ below.

Let $\Q(\C)={\{X_i|1\leq i\leq r}\}$ and $X_1=I$, where $r=\text{rank}(\C)$. Then   $6\leq r\leq9$ \cite[Lemma 4.2.2]{O3} \cite{BNRW},   Remark \ref{totallypost} says that there exist  positive integers $\alpha_i$ and $\beta_i$ such that $ \text{FPdim}(X_i)=\frac{\alpha_i+\beta_i\sqrt{5}}{2}$ for all $1\leq i\leq r$, thus
\begin{align}
15+5\sqrt{5}=1+\sum^r_{i=2} \text{FPdim}(X_i)^2=1+\sum^r_{i=2}\frac{\alpha_i^2+5\beta_i^2+2\alpha_i\beta_i\sqrt{5}}{4}.
\end{align}
Hence, $56=\sum^r_{i=2}\alpha_i^2+5\beta_i^2$ and $10=\sum^r_{i=2}\alpha_i\beta_i$. Without loss of generality, let  $\beta_i\leq\beta_j$ if $i\leq j$. Then a direct computation shows that there is a solution if and only if when $r=6,9$.

If $r=6$, then $(\alpha_2,\alpha_3,\alpha_4,\alpha_5,\alpha_6)=(1,1,1,3,2)$, $(\beta_2,\beta_3,\beta_4,\beta_5,\beta_6)=(1,1,1,1,2)$. If  $r=9$, then $(\beta_2,\beta_3,\beta_4,\beta_5,\beta_6,\beta_7,\beta_8,\beta_9)=(1,1,1,1,1,1,1,1)$ and there is a unique $X_j$ such that $\text{FPdim}(X_j)=\frac{3+\sqrt{5}}{2}$. From both cases, we know that $\C$ contains a self-dual simple object $Y$ of FP-dimension $\frac{1+\sqrt{5}}{2}$, thus $Y\otimes Y=I\oplus Z$ where $Z$ is also a self-dual simple object of FP-dimension $\frac{1+\sqrt{5}}{2}$. If $Z\cong Y$, then $\C$ contains a Yang-Lee fusion category $YL$ and $\C\cong YL\boxtimes \A$, where $\A$ is a modular fusion category of rank $3$ or $4$, \cite{RSW} says that $\A\cong\overline{YL}\boxtimes\C(\Z_2,\eta)$, this contradicts to the assumption $\C_\text{int}=\text{Vec}$. If $Z\ncong Y$, then $Y\otimes Z$ has to be simple by computing FP-dimensions of objects, while
\begin{align*}
Y\otimes(Y\otimes Z)=(Y\otimes Y)\otimes Z=Z\oplus Z\otimes Z,
\end{align*}
so $Y\otimes Z$ is the dual object of $Y$ since $Z$ is self-dual, it is impossible as they have different FP-dimensions. Therefore, $\C_\text{int}\neq\text{Vec}$ as claimed.
\end{proof}

\begin{theo}\label{moddim10} Let $\C$ be a   modular fusion category of global dimension $10$, then $\C$ is not simple.
\end{theo}
\begin{proof}
Let $\C$ be a  modular fusion category of global dimension $10$, then $\text{rank}(\C)=6,7,8,9$ \cite{BNRW} \cite[Lemma 4.2.2]{O3} and $\text{ord}(T_\C)$ divides $100$ by \cite[Corollary 8.18.2]{EGNO}. By \cite[Theorem II]{DongLinNg}
$\text{Gal}(\C)\subseteq\text{Gal}(\FQ(T_\C)/\FQ)\subseteq\text{Gal}(\FQ(\zeta_{10^2})/\FQ)$. Let $n_1$ and $n_2$ be the numbers of Galois conjugates of homomorphisms $\text{dim}(-)$ and $\text{FPdim}(-)$, respectively, and for any object $X$ we have $\text{FPdim}(X)=h_{Z_0}(X)$ for object $Z_0\in\Q(\C)$. Notice that simple objects   conjugated to $I$ have dimension $1$, then Siegel'trace theorem \cite[Theorem III]{Siegel} says that
 \begin{align}\label{traceequt}
 \text{dim}(\C)=\sum_{X\in\Q(\C)} \text{dim}(X)^2\geq n_1+\text{tr}(\text{dim}(Z_0)^2)> n_1+\frac{3}{2}n_2,
 \end{align}
unless $\text{dim}(Z_0)^2$ is a Galois conjugate of $\frac{3-\sqrt{5}}{2}$; here $\text{tr}(\text{dim}(Z_0)^2)$ is the sum of squared dimensions of simple objects that are  in the  orbit of $Z_0$. If $\text{dim}(Z_0)^2=\frac{3-\sqrt{5}}{2}$, then $\text{FPdim}(\C)=15+5\sqrt{5}$,   $\C$ is not simple by Lemma \ref{lemdim10}, however.
Assume that $\text{dim}(Z_0)^2\neq\frac{3\pm\sqrt{5}}{2}$ below.

Note that $\text{Gal}(\FQ(\zeta_{10^2})/\FQ)\cong\Z_2\times\Z_4\times\Z_5$, and $\FQ(\zeta_{10^2})$  has exactly three quadratic subfields $\FQ(\sqrt{5})$, $\FQ(\sqrt{-5})$ and $\FQ(\sqrt{-1})$, so $\text{Gal}(\FQ(\zeta_{10^2}+\zeta^{-1}_{10^2})/\FQ)\cong\Z_4\times\Z_5$ as $\FQ(\zeta_{10^2}+\zeta^{-1}_{10^2})$ has a unique quadratic subfield $\FQ(\sqrt{5})$.
Since homomorphisms $\text{dim}(-)$ and $\text{FPdim}(-)$ takes values in $\FQ(\zeta_{10^2}+\zeta^{-1}_{10^2})$, $n_1=2,4,5$.
If $n_1=4,5$, then $n_2=2$ by equation \ref{traceequt},  and we obtain a quadratic equation $x^2-bx+a=0$ with $\sqrt{b^2-4a}=m\sqrt{5}$ for some positive integer $m$, where  $a$ is the product of $\text{FPdim}(\C)$ and its Galois conjugate, $b$ is  the sum of $\text{FPdim}(\C)$ and its Galois conjugate. Thus, $a$ divides both $100$ and $b^2$ by \cite[Lemma 2.7]{O1} and $b>10$, and equation \ref{classequation} means that $b/a\leq3/5$, it is easy to see  that there is no solution except $\text{FPdim}(\C)=15+5\sqrt{5}$.

If $n_1=2$, then $n_2=2$ or $4$ by equation \ref{traceequt}. Since $\FQ(\zeta_{10^2}+\zeta^{-1}_{10^2})$ has a unique quadratic subfield  $\FQ(\sqrt{5})$,  $\text{dim}(Z_0)\in\FQ(\sqrt{5})$, so $\text{FPdim}(\C)\in\FQ(\sqrt{5})$ as $\text{FPdim}(\C)=\frac{10}{\text{dim}(Z_0)^2}$. If $n_2=4$,  let $\mathbb{K}=\FQ(\text{FPdim}(X): \forall X\in\Q(\C))$, then $\C$ is faithfully graded by $\text{Gal}(\mathbb{K}/\FQ(\text{FPdim}(\C))$ by \cite[Proposition 1.8]{GS}, which must be  a non-trivial group as $\text{FPdim}(-)$ has four Galois conjugates. Thus $\C$ contains a non-trivial fusion subcategory. If $n_2=2$, then $\mathbb{K}=\FQ(\sqrt{5})$ and one can easily obtain that $\text{FPdim}(\C)=15+5\sqrt{5}$ by previous arguments. Hence, modular fusion categories of  dimension $10$ are not simple.
\end{proof}

\begin{coro}\label{corodimen10}Let $\C$ be a pre-modular fusion category and $\text{dim}(\C)=10$. Then as  a pre-modular fusion category,  $\C$ is pointed,  or $\C\cong (YL\boxtimes \overline{YL})^{\Z_2}$, or $\C\cong (YL\boxtimes \overline{YL})\boxtimes\D$, or $\C$ has the same fusion rules as $\text{Rep}(D_{10})$, where $\D$ is a pointed  fusion category of global dimension $2$, $D_{10}$ is the dihedral group of order $10$.
\end{coro}
\begin{proof} If   $\C$ is symmetric, then $\text{FPdim}(\C)=10$, so either $\C\cong \text{Rep}(\Z_{10})$ or $\C\cong \text{Rep}(D_{10})$ as symmetric fusion category. Assume that $\C$ is not symmetric below, it follows from Theorem \ref{Lagrange} that $\text{dim}(\C')=1, 2$ or $5$.  When $\text{dim}(\C')=5$, then $\C'=\text{Rep}(\Z_5)$ is a  Tannakian fusion category. Hence  $\C_{\Z_5}\cong\C(\Z_2,\eta)$ is a modular fusion category \cite[Proposition 4.30]{DrGNO2}, and $\C\cong \text{Rep}(\Z_5)\boxtimes\C(\Z_2,\eta)$. When $\text{dim}(\C')=2$, then $\C'$ is braided equivalent to $\text{Rep}(\Z_2)$ or $\text{sVec}$. In the first case,  $\C\cong\B^{\Z_2}$, where $\B$ is a  modular fusion category of global dimension $5$. Hence \cite[Example 5.1.2]{O3} says that  $\C\cong (YL\boxtimes \overline{YL})^{\Z_2}$  or $\C\cong (YL\boxtimes \overline{YL})\boxtimes \text{Rep}(\Z_2)$ if $\B\cong YL\boxtimes\overline{YL}$,  and  $\C$ is pointed or $\C$ has the same fusion rules as $\text{Rep}(D_{10})$ if $\B\cong\C(\Z_5,\eta)$. If  $\C'\cong \text{sVec}$, then   it is exactly the conclusion of Lemma \ref{supdimen10}.

If $\C'=\text{Vec}$,  let $\A$ be an arbitrary non-trivial fusion subcategory of $\C$, since $\C$ is not simple by Theorem \ref{moddim10}. Notice that $\A\cap\A_\C'$ is a symmetric fusion subcategory, where $\A_\C'$ is the centralizer of $\A$ in $\C$,  $\text{dim}(\A\cap\A_\C')^2$ divides $10$ by Theorem \ref{Lagrange}. So $\A\cap\A_\C'=\text{Vec}$,  and $\C\cong\A \boxtimes\A_\C'$ as modular fusion category  \cite[Theorem 3.13]{DrGNO2}. Meanwhile $\text{rank}(\A)\text{rank}(\A_\C')=\text{rank}(\C)$ and $\text{rank}(\C)\leq10$
by \cite[Lemma 4.2.2]{O3}, then classification follows from \cite{BNRW,RSW}.
\end{proof}

\subsection{Spherical fusion categories of global dimension   $6$}\label{subsection4.3}
In this subsection, we   show that spherical fusion category of global dimension $6$ is weakly integral, which generalizes Proposition \ref{predimen6}.  The proof is similar to the argument of Theorem \ref{moddimen7}.

\begin{theo}\label{spherical6}Spherical fusion categories of global dimension $6$ are weakly integral.
\end{theo}
\begin{proof} Let $\C$ be spherical fusion category of global dimension $6$. It follows from
\cite[Remark 4.2.3]{O3} that $\C$ is pointed if and only if $\C$ has rank $6$. Moreover, if $\text{rank}(\C)=3$, this is   \cite[Theorem 1.1]{O2}. Below we assume that $\C$ is not weakly integral,   in particular, homomorphisms $\text{dim}(-)$ and $\text{FPdim}(-)$ are not in the same orbit under action of Galois group $\text{Gal}(\overline{\mathbb{Q}}/\mathbb{Q})$.
Hence, $\text{Gr}(\C)$ is commutative and $\text{rank}(\C)=4,5$. Note that homomorphisms from $\text{Gr}(\C)$ to $\FC$ are divided into two  or three orbits under the action of Galois group $\text{Gal}(\overline{\mathbb{Q}}/\mathbb{Q})$. Moreover, $\text{dim}(\C_\text{int})=\text{FPdim}(\C_\text{int})$ equals to $1,2,3$ by Theorem \ref{Lagrange}, so $\C_\text{int}=\C_\text{pt}$. Here, we only give a proof when $\text{rank}(\C)=4$, the other is same.
We know that orbits of $\text{dim}(-)$ and $\text{FPdim}(-)$ contain    $2$ homomorphisms
 \cite[Exercise 9.6.2]{EGNO}. Let $f_2$ be   the non-trivial Galois conjugate of $f_1=\text{FPdim}(\C)$. Then equation \ref{classequation} says  that
\begin{align*}
\frac{1}{3}+\frac{1}{f_1}+\frac{1}{f_2}=1,
\end{align*}
 in addition, $f_1f_2|36$ by \cite[Proposition 8.22]{ENO1} and $f_1+f_2>6$. Hence, $f_1f_2=12,18,36$.
 If $f_1f_2=12$, then $f_1=6$ and $f_2=2$, this is impossible.

If $f_1f_2=18$, then $f_1=6+3\sqrt{2}$ and $f_1=6-3\sqrt{2}$.
If $\text{FPdim}(\C_\text{int})=3$, then there is  a unique simple object $X$ such that $\text{FPdim}(X)^2=3+3\sqrt{2}$. Let $G(\C)=\Z_3=\langle g\rangle$, then $g\otimes X=X$. However,  $X\otimes X=I\oplus g\oplus g^2 \oplus n X$, so $n\text{FPdim}(X)=3\sqrt{2}$, this is impossible. If $\text{FPdim}(\C_\text{int})=2$, then Remark \ref{totallypost} says  that there exist positive integers $\alpha_i,\beta_i$ such that
\begin{align}
6+3\sqrt{2}=2+\frac{a_1+b_1\sqrt{2}}{2}+\frac{a_2+b_2\sqrt{2}}{2}
=2+(\frac{\alpha_1+\beta_1\sqrt{2}}{2})^2+(\frac{\alpha_2+\beta_2\sqrt{2}}{2})^2
\end{align}
is the sum of of simple objects of $\C$.  Assume $b_1\geq b_2$. Then $(b_1,b_2)=(5,1),(4,2)$, $(3,3)$, there does not exists such $(\alpha_i,\beta_i)$, however. If $\text{FPdim}(\C_\text{int})=1$, then \begin{align}
6+3\sqrt{2}=1+\sum_{i=1}^3\frac{a_i+b_i\sqrt{2}}{2}
=1+\sum_{i=1}^3(\frac{\alpha_i+\beta_i\sqrt{2}}{2})^2
\end{align}
is the sum of FP-dimensions of simple objects of $\C$, where $b_3\geq b_2\geq b_1$ and $\alpha_i,\beta_i$ are positive integers  by Remark \ref{totallypost}. Thus, $(b_1,b_2,b_3)=(1,1,4)$, or $(1,2,3)$, or $(2,2,2)$.
By a direct computation we see that there does not exist such integers $\alpha_i,\beta_i$. The subcase $f_1f_2=36$ can be proved by using similar arguments.

Hence, spherical fusion categories $\C$ of global dimension $6$ are always weakly integral.
Thus $\C$ is either pointed, or $\C$ has the same fusion rules as $\text{Rep}(S_3)$, or $\C_\text{pt}\cong\text{Vec}_{\Z_3}$ and $\C$ contains a unique simple object of FP-dimension $\sqrt{3}$,
see \cite[Theorem 1.1]{EGO} for details.
\end{proof}

\section*{Acknowledgements}
 The author is grateful to   V. Ostrik for  insightful conversations on \'{e}tale algebras and totally positive algebraic integers, particularly  for providing reference \cite{KiO}. And the author also thanks  Y. Wang for helps in understanding the Galois symmetry of modular categories. The author thanks the anonymous referee for   numerous suggestions that helped improve this paper substantially.  Part of this paper was written during a visit of the   author at University of Oregon supported by China Scholarship Council (grant No. 201806140143), he  appreciates the Department of Mathematics for their warm hospitality.

\section*{Data Availability Statement}
Data sharing not applicable to this article as no datasets were generated or analysed during the current study.

\bigskip
\author{Zhiqiang Yu\\ \thanks{Email:\,zhiqyumath@yzu.edu.cn}
\\{\small School of Mathematical Sciences,  Yangzhou University, Yangzhou 225002, China}}

\begin{thebibliography}{50}
\bibitem[1]{BK} B. Bakalov and A. Kirillov, Jr, Lectures on tensor categories and modular functors,
 University Lecture Series $\mathbf{21}$, Amer. Math. Soc.,  2001.
 \bibitem[2]{BGNPRW}P. Bruillard, C. Galindo, S-H. Ng, J. Plavnik, E. Rowell and Z. Wang, Classification of super modular categories by rank, Algebr. Represent. Theory, $\mathbf{23}$ (2020), 795-809.
\bibitem[3]{BNRW}P. Bruillard, S.-H. Ng, E. Rowell and Z. Wang, On classification of modular categories by rank, Int. Math. Res. Not. $\mathbf{2016}$ (2016), no. 24, 7546-7588.
\bibitem[4]{BPRZ}P. Bruillard, J. Plavnik, E. Rowell and Q. Zhang, On classification of super-modular categories of rank 8, J. Algebra. Appl, to appear, https://doi.org/10.1142/S021949882140017X; arXiv:1909.09843.
\bibitem[5]{DMNO}A. Davydov, M. M\"{u}ger, D. Nikshych and V. Ostrik, The  Witt group of  non-degenerate  braided fusion categories, J. Reine. Angew. Math. $\textbf{677}$ (2013), 135-177.
\bibitem[6]{DongLinNg}  C. Dong, X. Lin and S-H. Ng, Congruence property in conformal field theory,  Algebra Number Theory $\mathbf{9}$ (2015), no. 9, 2121-2166.
\bibitem[7]{DrGNO2}V. Drinfeld, S. Gelaki, D. Nikshych and V. Ostrik, On braided fusion categories \uppercase\expandafter{\romannumeral 1}, Sel. Math. New. Ser. $\textbf{16}$ (2010), no. 2, 1-119.
\bibitem[8]{EGNO}P. Etingof, S. Gelaki, D. Nikshych and V. Ostrik, Tensor categories, Mathematical Surveys and Monographs $\textbf{205}$, Amer. Math. Soc., 2015.
\bibitem[9]{EGO}P. Etingof, S. Gelaki  and V. Ostrik, Classification of fusion  categories of dimension $pq$, Int. Math. Res. Not. $\textbf{2004}$ (2004), no. 57, 3041-3056.
\bibitem[10]{ENO1}P. Etingof, D. Nikshych and V. Ostrik, On fusion categories, Ann. of Math. $\textbf{162}$ (2005), no. 2, 581-642.
\bibitem[11]{GS} T. Gannon and A. Schopieray,   Algebraic number fields generated by Frobenius-Perron dimensions in fusion rings, arXiv:1912.12260.
\bibitem[12]{Green}D. Green,  Classification of rank $6$ modular   categories with  Galois group $\langle (012)(345)\rangle$, arXiv:1908.07128.
\bibitem[13]{KiO} A. Kirillov, Jr and V. Ostrik, On a $q$-analogue of the McKay correspondence and the ADE classification of $\mathfrak{sl}_2$ conformal field theories, Adv. Math. $\mathbf{171}$ (2002), no. 2, 183-227.
\bibitem[14]{Lusztig}G. Lusztig, Hecke algebras with unequal parameters, CRM Monograph Series, $\mathbf{18}$. Amer. Math. Soc.,  2003.
\bibitem[15]{Mu}M. M\"{u}ger, Galois theory for braided tensor categories and the modular closure, Adv. Math. $\mathbf{150}$ (2000), no. 2, 151-201.
\bibitem[16]{Mu1}M. M\"{u}ger, From subfactors to categories and topology I. Frobenius algebras in and Morita equivalence of tensor categories, J. Pure Appl. Algebra. $\mathbf{180}$ (2003), no. 1-2, 81-157.
\bibitem[17]{Mu2}M. M\"{u}ger, From subfactors to categories and topology II. The quantum double of tensor categories and subfactors, J. Pure Appl. Algebra. $\mathbf{180}$ (2003), no. 1-2, 159-219.
\bibitem[18]{NS}  S.-H. Ng and P. Schauenburg, Frobenius-Schur indicators and exponents of spherical
categories, Adv. Math. $\mathbf{ 211}$ (2007), no. 1, 34-71.
\bibitem[19]{O1} V. Ostrik, On formal codegrees of fusion categories, Math. Res. Lett. $\mathbf{16}$ (2009), no. 5, 899-905.
\bibitem[20]{O2}V. Ostrik, Pivotal fusion categories of rank 3, Mosc. Math. J. $\mathbf{15}$ (2015), no. 2, 373-396.
\bibitem[21]{O3}V. Ostrik, Remarks on global dimension of  fusion categories,  Tensor categories and Hopf algebras, 169-180, Contemp. Math. $\mathbf{728}$, Amer. Math. Soc., 2019.
\bibitem[22]{RSW}E. Rowell, R. Stong  and Z. Wang, On classification of modular tensor categories,
Comm. Math. Phy. $\mathbf{292}$ (2009), no. 2, 343-389.
\bibitem[23]{Sch}A. Schopieray, Norm, trace and formal codegrees of fusion categories, J. Algebra, $\mathbf{568}$ (2021), 362-385.
\bibitem[24]{Siegel}S. Siegel, The trace of totally positive and real algebraic integers, Ann. of Math. $\textbf{46}$ (1945), no. 2, 302-312.
\bibitem[25]{Yu}Z. Yu, On slightly degenerate fusion categories, J. Algebra, $\mathbf{559}$ (2020), 408-431.
\end{thebibliography}
\end{document}